\documentclass[draft]{amsart}
\usepackage[english]{babel}
\usepackage[normalem]{ulem}
\usepackage[all]{xy}
\usepackage{amssymb}
\usepackage{mathrsfs} 
\usepackage{dashrule}
\numberwithin{figure}{section}
\numberwithin{equation}{section}
\usepackage{tikz}
\usetikzlibrary{shapes.geometric}
\usepackage{comment}
\usepackage{graphicx}
\usepackage{float}
\usepackage{xcolor}
\usepackage{xspace}
\usepackage{enumerate}
\usepackage{caption}
\usepackage{subcaption}
\usepackage{tkz-euclide}
\usepackage[utf8]{inputenc}
\usetkzobj{all}
\definecolor{ruta}{rgb}{0.309, 0.459, 0.208}
\definecolor{ruta1}{rgb}{0.205,0.301,0.201}
\definecolor{ruta2}{rgb}{0.409, 0.459, 0.208}
\definecolor{vino}{rgb}{0.256,0,0}
\definecolor{siva}{rgb}{0.205,0.201,0.201}
\definecolor{siva'}{rgb}{0.250,0.116,0.116}

\let\cal\mathcal

\def\Hscr{{\cal H}}

\def\Kscr{{\cal K}}
\def\Lscr{{\cal L}}
\def\Mscr{{\cal M}}

\def\Oscr{{\cal O}}

\def\Sscr{{\cal S}}

\let\blb\mathbb

\def \ZZ{{\blb Z}}
\def \NN{{\blb N}}
\def \RR{{\blb R}}

\def\Lotimes{\overset{L}{\otimes}}

\def\Spec{\operatorname {Spec}}

\def\Ext{\operatorname {Ext}}
\def\Hom{\operatorname {Hom}}

\def\End{\operatorname {End}}
\def\RHom{\operatorname {RHom}}

\def\relint{\operatorname {relint}}

\def\End{\operatorname {End}}

\def\gldim{\operatorname {gl\,dim}}

\def\r{\rightarrow}

\def\pdim{\operatorname{pdim}}

\newcommand{\la}{\langle}
\newcommand{\ra}{\rangle}

\newcommand{\D}{{\rm D}}

\newtheorem{lemma}{Lemma}[section]
\newtheorem{proposition}[lemma]{Proposition}
\newtheorem{theorem}[lemma]{Theorem}
\newtheorem{corollary}[lemma]{Corollary}

\theoremstyle{definition}

\newtheorem{definition}[lemma]{Definition}
\newtheorem{conjecture}[lemma]{Conjecture}

{

}

\theoremstyle{remark}

\newtheorem{remark}[lemma]{Remark}

\newdimen\uboxsep \uboxsep=1ex
\def\uboxn#1{\vtop to 0pt{\hrule height 0pt depth 0pt\vskip\uboxsep
\hbox to 0pt{\hss #1\hss}\vss}}

\def\uboxs#1{\vbox to 0pt{\vss\hbox to 0pt{\hss #1\hss}
\vskip\uboxsep\hrule height 0pt depth 0pt}}

\def\codim{\operatorname{codim}}

\let\oldmarginpar\marginpar
\def\marginpar#1{ \oldmarginpar{\tiny \raggedright #1}}

\def\Sym{\operatorname{Sym}}

\def\T{\Mscr}
\def\TT{\mathfrak{M}}

\title[On the noncommutative Bondal-Orlov conjecture]{On the noncommutative Bondal-Orlov conjecture for some toric varieties}

\author[\v{S}pela \v{S}penko and Michel Van den Bergh]{\v{S}pela \v{S}penko and Michel Van den Bergh\\ with an appendix by Jason P. Bell}
\thanks{The first author is a FWO $[$PEGASUS$]^2$ Marie Sk\l odowska-Curie fellow at the Free University of Brussels
(funded by the European Union Horizon 2020 research and innovation
programme under the Marie Sk\l odowska-Curie grant agreement
No 665501 with the Research Foundation Flanders (FWO)). During part of this work she was also a postdoc with Sue Sierra at the University of Edinburgh.}
\address{Departement Wiskunde, Vrije Universiteit Brussel, 
Pleinlaan $2$, B-1050 Elsene}
\email[]{spela.spenko@vub.ac.be}
\address{Departement WNI, Universiteit Hasselt, Universitaire Campus \\
B-3590 Diepenbeek}
\email[]{michel.vandenbergh@uhasselt.be}
\thanks{The second author is a senior researcher at the Research Foundation Flanders (FWO).  While working on this project he was supported by
the FWO grant G0D8616N: ``Hochschild cohomology and deformation theory of triangulated categories''.}
\address{Department of Pure Mathematics, University of Waterloo, ON, Canada N2L 3G1}
\email[]{jpbell@uwaterloo.ca}

\begin{document}

\begin{abstract}
  We show that all toric noncommutative crepant resolutions (NCCRs) of
 affine GIT quotients of ``weakly symmetric'' unimodular
torus representations are derived equivalent. This
  yields evidence for a non-commutative extension of a well known
  conjecture by Bondal and Orlov stating that all crepant resolutions
of a Gorenstein singularity are derived equivalent.  We prove our result by showing
  that \emph{all} toric NCCRs of the affine GIT quotient are derived equivalent to a fixed Deligne-Mumford GIT
  quotient stack associated to a generic character of the torus. This extends a result by Halpern-Leistner and Sam
  which showed that such GIT quotient stacks are a geometric
  incarnation of a family of \emph{specific} toric NCCRs constructed
  earlier by the authors.
\end{abstract}

\maketitle

\section{Introduction}
The
Bondal-Orlov conjecture \cite{BonOrl} asserts that all crepant 
resolutions of Gorenstein singularities are derived equivalent. 
Later the
conjecture was further generalized to a noncommutative setting.

\begin{definition}
Let $S$ be a normal noetherian Gorenstein domain. A {\em noncommutative crepant resolution} (NCCR) of $S$ is an $S$-algebra of a finite global dimension of the form $\End_{S}(M)$, which is Cohen-Macaulay as an $S$-module, where $M$ is a nonzero finitely generated reflexive $S$-module.
\end{definition}

In \cite[Conjecture 4.6]{VdB04} it is then conjectured that crepant resolutions should also be equivalent to  noncommutative crepant resolutions. In \cite[Conjecture 1.4]{IyamaWemyssNCBO} the noncommutative part is singled out as ``noncommutative Bondal-Orlov''.

\begin{conjecture}\label{conj}\cite{VdB04,IyamaWemyssNCBO}
All noncommutative crepant resolutions are derived equivalent.
\end{conjecture}

NCCRs do not always exist, however in \cite{SVdB} they were
constructed for a large class of quotient singularities for reductive
groups.  Notably they exist (under mild genericity condition, see
Definition \ref{def:generic}) for quotient singularities for
quasi-symmetric representations $W$ of a torus $T$. Here
``quasi-symmetric'' means that the sum of weights of $W$ on each line
through the origin is zero (see \cite{Kite} for a geometric
interpretation of quasi-symmetric representations).  The NCCRs in
loc.\ cit.\ are given by \emph{modules of covariants} $M=(U \otimes_k
\Sym(W))^T$ for a representation $U$ of $T$, and in this setting (as
$T$ is a torus and hence $U$ is a sum of characters) are called {\em
  toric}.

In this note we prove Conjecture \ref{conj} for toric NCCRs in the above context(thus in particular the existence of toric NCCRs is guaranteed). In fact, our assumptions are slightly weaker and apply to
``weakly symmetric'' (generalizing quasi-symmetric)
representations, see Definition \ref{def:quasisym} (for which however
toric NCCRs do not always exist, see \cite[Example 10.1]{SVdB}).  The
following is our main
result. 
\begin{theorem}\label{thm}
Let $W$ be a unimodular, generic and weakly symmetric representation $W$ of a torus $T$. Then all toric NCCRs 
of $\Sym(W)^T$ are derived equivalent.
\end{theorem}

To prove the theorem we follow a strategy of Halpern-Leistner and Sam \cite{HLSam} and embed any toric NCCR  of $\Sym(W)^T$ into the derived category of a \emph{fixed} generic GIT stack quotient of $W^*=\Sym W$.
Such GIT quotient stacks are Calabi-Yau and hence they do not admit non-trivial semi-orthogonal decompositions. Therefore the constructed  embedding is in fact an equivalence, proving the theorem.

We note that Halpern-Leistner and Sam already showed that generic GIT stack quotients of $W^\ast$ for quasi-symmetric $W$ are derived equivalent to the specific toric 
 NCCRs constructed in \cite{SVdB} (already mentioned above).
But since we start with an arbitrary toric NCCR we are dealing
 with a more general class of Cohen-Macaulay modules than in
 \cite{HLSam} and therefore the construction of the embedding is
 more intricate and requires the use of the
 Cohen-Macaulayness criterion for modules of covariants from \cite{Vdb1}. Luckily this criterion turns out to considerably simplify in the weakly symmetric case.

\medskip
In \S\ref{app} we present an alternative proof of Theorem \ref{thm} in the case $T=G_m$. 
It is a based on a crucial combinatorial lemma provided by Jason Bell in Appendix \ref{sec:Jason}, which is used to 
 describe  ``maximal Cohen-Macaulay cliques'' of $X(T)$  (see \eqref
{1}).
Moreover, we introduce the notion of {\em toric maximal modification algebras} (toric MMAs)  (see the paragraph after Definition \ref{def:MMA}) and show that in the case $T=G_m$ they coincide with NCCRs (see Proposition \ref{prop:MMA}).

\medskip

We note that Iyama and Wemyss \cite{IyamaWemyssNCBO} provide a sufficient criterion under which Conjecture \ref{conj} holds, which in particular covers dimension $\leq 3$. However, this criterion does not seem to be easily applicable to our setting. 

\section{Acknowledgement}
We thank J\o rgen Vold Rennemo for interesting and useful discussions.

\section{Notation and conventions}
\label{sec:notations}
Throughout $k$ is an algebraically closed field of characteristic $0$. 
Let $W$ be a $d$-dimensional $T$-representation, $\dim T=s$, and let $R=\Sym W$,  $X=\Spec R=W^*$. Let $(\alpha_i)_{i=1}^d$ be the weights of $W$ and let $(w_i)_i$
be the corresponding weight vectors. 

\medskip

We denote by $X(T)$ (resp. $Y(T))$ the character group (resp. the group of one-parameter subgroups) of $T$. 
There is a natural pairing $Y(T)\times X(T)\to \ZZ$, which extends to $Y(T)_\RR\times X(T)_\RR\to \RR$, we denote it by $\la\ \!,\ \!\!\ra$. On $Y(T)_\RR$ we choose a positive definite quadratic form, and we denote the corresponding norm by $\|\;\|$. 

If $U$ is a $1$-dimensional representation of $T$ given by
$\mu\in X(T)$ then we write $M(\mu)$ instead of $M(U)$. In our main
reference \cite{Vdb1} the results are written in terms of
semi-invariants $R_\mu^T$, defined as the sum of all irreducible
representations of $T$ in $R$ with character $\mu$; i.e.
$R_\mu^T\cong M(-\mu)$. We will use both notations, depending on the
context. More generally, if $M$ is any $T$-module, we write $M_\mu^T$ for the sum
of all irreducible representations of $T$ in $M$ with character $\mu$.

\medskip 

Let us denote 
\begin{align*}
X^{\lambda,+}&=\{x\in X\mid  \lim_{t\to 0}\lambda(t)x \text{ exists}\},\\
X^{\lambda,>0}&=\{x\in X\mid  \lim_{t\to 0}\lambda(t)x=0\},
\end{align*}
and let $W^{\lambda,>}$, $W^{\lambda,+}$  be the subspace of $W$ generating the defining ideal $I^{\lambda,>}$, $I^{\lambda,+}$  of $X^{\lambda,>0}$ and $X^{\lambda,+}$. Note that 
these subspaces are spanned by the weight vectors $w_i$ such that $\la\lambda,\alpha_i\ra\geq 0$, $\la \lambda,\alpha_i \ra >0$, respectively.

We write $X^u=\{x\mid 0\in \overline{Tx}\}$ for the $T$-unstable locus 
(also called ``nullcone'') of~$X$. A defining ideal for $X^u$ is $R(R^T)^+$ where $(R^T)^+$ is the augmentation
ideal of~$R^T$.
By the Hilbert-Mumford criterion, $X^u=\cup_{\lambda\in Y(T)}X^{\lambda,>0}$.
Let $\chi\in X(T)$. Then by definition $X^{ss,\chi}$ consists of the points $x\in X$ such that if $\lambda\in Y(T)$
is such that $\lim_{t\r 0} \lambda(t)x$ exists then $\langle \lambda,\chi \rangle\ge 0$. We write $X^{u,\chi}:=X\setminus X^{ss,\chi}=\bigcup_{\lambda:\langle
\lambda,\chi\rangle< 0}X^{\lambda,+}$ and call it the $\chi$-unstable locus.

\medskip

For further reference we introduce some extra notation, mostly 
consistent with \cite{Vdb1}: $T_\lambda=\{i\mid \la\lambda,
\alpha_i\ra<0\}$,
$d_\lambda=\operatorname{codim}(X^{\lambda,>0},X)=d-|T_\lambda|$.
Moreover, $T_\lambda^0=\{i\mid \la\lambda, \alpha_i\ra=0\}$,
$d_\lambda^0=|T_\lambda^0|$. $T_\lambda^+=\{i\mid \la\lambda, \alpha_i\ra>0\}$.
On $Y(T)_\RR$ we let $\lambda\sim
\lambda'$ iff $T_\lambda=T_{\lambda'}$, $B=\{\lambda\in
Y(T)_\RR\mid \|\lambda\|<1\}$, $\Lambda=B/{\sim}$,
$B_\lambda=\{\mu\in B\mid \mu\sim \lambda\}$,
$\Phi_\lambda=\bar{B}_\lambda-{B}_\lambda$,
$H_\lambda={\rm span}\{(\alpha_i)_{i\in T_\lambda^0}\}$, 
$h^0_\lambda=\dim H_\lambda$. 

\medskip

We  introduced some properties
\begin{definition}\label{def:generic}
A $T$-representation $W$ is {\em generic} if the set $X^s$ of points in $X$ with closed orbit and trivial stabilizer is nonempty and satisfies $\codim(X-X^s)\geq 2$ (equivalently for all $\lambda\in Y(T)\setminus \{0\}$  there exist two $1\leq i\leq d$ such that $\la\lambda,\alpha_i\ra>0$).
\end{definition}
\begin{definition}\label{def:quasisym}
A $T$-representation $W$ is {\em quasi-symmetric} if for every line $\ell\subset X(T)_\RR$ through
the origin we have $\sum_{\alpha_i\in\ell}\alpha_i=0$. It is \emph{weakly symmetric} 
if for every $\ell$ the cone spanned by ${\alpha_i\in\ell}$ is either zero or $\ell$.
\end{definition}
We denote $\Sigma:=\{\sum_i a_i\alpha_i\mid a_i\in ]-1,0]\}$. 
\begin{definition}
\label{def:generic}
We say that $\chi\in X(T)$ is {\em generic for $W$} if it is parallel to $\Sigma$ but not parallel to any face of $\Sigma$.
\end{definition}

\section{Local cohomology  in the weakly symmetric case}
As alluded to in the introduction we will construct 
an embedding of the derived category of a toric NCCR
$\Lambda=\End_{R^T}(\oplus_\gamma M(\gamma))$, into the derived category of a fixed 
GIT quotient stack of the form $X^{ss,\chi}/T$. To relate
the Cohen-Macaulayness of $\Lambda$ to the necessary vanishing on 
$X^{ss,\chi}/T$ (see  Proposition \ref{prop:git} and its proof) 
it will turn out that we need to compare (as $T$-modules) the local cohomology of 
$\Oscr_X$ supported in $X^u$  and in $X^{u,\chi}$.
We have a good understanding of the former by \cite{Vdb1} and  we employ the HKKN\footnote{Hesselink-Kempf-Kirwan-Ness} stratification \cite{Kirwan} for the 
latter.
\subsection{Support in the nullcone}\label{nullcone}
In the first part of this section $W$ is arbitrary; i.e. not necessarily weakly symmetric or generic. 
The local cohomology modules of $\Oscr_X$ supported in the nullcone  
 are described in \cite{Vdb1}. In particular, they provide a criterion for Cohen-Macaulayness of modules of covariants. At the end of this section we show that 
when $W$ is weakly symmetric and generic the criterion becomes more concrete.
\begin{lemma}\cite[Theorem 3.4.1]{Vdb1}\label{lemma}
Let $I=R(R^T)^+$ be the defining ideal of $X^u$ (see \S\ref{sec:notations}). Then
\begin{equation}\label{Gequi}
H^i_{(R^T)^+}(R_\mu^T)\cong H^i_I(R)^T_\mu, 
\end{equation}
and hence $R_\mu^T$ is Cohen-Macaulay if and only if $H^i_I(R)^T_\mu=0$ for $0\leq i<\dim R^T$.
\end{lemma}

\begin{theorem}\cite[Proposition 3.3.1, Theorem 3.4.1
]{Vdb1}\label{theorem}
There is a $T$-equivariant filtration on $H^i_{X^u}(X,\Oscr_X)$ together with a $\ZZ^d$-graded isomorphism of $R$-modules
\begin{equation}\label{loccoh}
{\rm gr}\, H^i_{X^u}(X,\Oscr_X)\cong\bigoplus_{\lambda\in \Lambda}\tilde{H}^{i+s-d_\lambda-1}(\Phi_\lambda,k)\otimes H^{d_\lambda}_{X^{\lambda,>0}}(X,\Oscr_X).
\end{equation}
Furthermore there is a $\ZZ^d$-graded isomorphism 
\begin{equation}\label{loccohstrata}
\begin{aligned}
H^{d_\lambda}_{X^{\lambda,>0}}(X,\Oscr_X)&\cong 
(\wedge^{d_\lambda} W^{\lambda,>})^*\otimes \bigoplus_t \Sym^t (W^{\lambda,>})^*\otimes R/I^{\lambda,>}\\
&\cong (\wedge^{d_\lambda} W^{\lambda,>})^*\otimes \bigoplus_t \Sym^t 
(W^{\lambda,>\ast}\oplus W/W^{\lambda,>0})
\end{aligned}
\end{equation}
\end{theorem}
In our situation the $\Phi_\lambda$ are easy to describe.
\begin{lemma} \label{lem:homeo}
Assume that $T$ acts faithfully on $W$ and that $W$ is weakly symmetric. Let $\lambda\in Y(T)_\RR$.
Then $\Phi_\lambda$
is homeomorphic with a sphere (of maximal dimension) in $H_\lambda^\perp\subset Y(T)_\RR$.
\end{lemma}
\begin{proof}
We claim that
\[
B_\lambda=\{\mu\in B\cap H^\perp_\lambda\mid \forall i\in T_\lambda:\langle \mu,\alpha_i\rangle<0\}. 
\]
Indeed, if $\mu\in B\cap H^\perp_\lambda$ then
$T_{\lambda}^0\subset T_\mu^c$, and if $T_\lambda\subset T_\mu$ then
$T_{-\lambda}\subset T_\mu^c$ by the weak symmetricity,
respectively. As
$\{i\mid 1\leq i\leq d\}=T_\lambda\sqcup T_\lambda^0\sqcup
T_{-\lambda}$
we have $T_\mu=T_\lambda$ and thus $\mu\in B_\lambda$. For
the converse it is enough to observe that $\mu\in  B_\lambda$
implies $\mu\in H_\lambda^\perp$. Let $\mu\in B_\lambda$ and write
$\mu=\mu_0+\mu_\perp$ for $\mu_0\in H_\lambda$,
$\mu_\perp\in H_\lambda^\perp$. If $\mu_0\neq 0$ then
$\la\mu_0,\alpha_i\ra<0$ for some $i\in T_\lambda^0$ since $W$ is
weakly symmetric. Thus, $\mu\in H_\lambda^\perp$. 

\medskip

 Let
$\Gamma\subset X(T)_\RR/H_\lambda$ be the cone spanned by the images
of the weights $(-\alpha_i)_{i\in T_\lambda}$ in $X(T)_\RR/H_\lambda$. Note that
by faithfulness and weak symmetry $\Gamma$ is of maximal dimension.
$\langle-,-\rangle$ descends to a non-degenerate pairing between
$H^\perp_\lambda$ and $X(T)_\RR/H_\lambda$.  Let
$\Gamma^\vee \subset H^\perp_\lambda$ be the dual cone of $\Gamma$.
Then we have
\begin{align*}
B_\lambda&=\{\mu\in B\cap H^\perp_\lambda\mid \mu{|}(\Gamma-\{0\})>0\} \\
&=B\cap \relint \Gamma^\vee.
\end{align*}
The conclusion now easily follows.
\end{proof}
\begin{proposition} \label{prop:weakly}
Assume $T$ acts faithfully and $W$ is weakly symmetric. Then
\[
{\rm gr}\, H^i_{X^u}(X,\Oscr_X)\cong\bigoplus_{\lambda\in \Lambda,i=d_\lambda-h^0_\lambda} H^{d_\lambda}_{X^{\lambda,>0}}(X,\Oscr_X).
\] 
\end{proposition}
\begin{proof} 
 By Lemma \ref{lem:homeo} $\Phi_\lambda$ is a sphere of dimension $s-1-h^0_\lambda$.
Hence the non-zero terms in \eqref{Gequi} correspond to 
$
i+s-d_\lambda-1=s-1-h^0_\lambda
$
or
$
i=d_\lambda-h^0_\lambda
$.
\end{proof}
\begin{corollary}\label{decor} Assume $W$ is generic and weakly symmetric. 
Then $R^T_\mu$ is Cohen-Macaulay if and only if for all $0\neq \lambda\in Y(T)$ we have
$H^{d_\lambda}_{X^{\lambda,>0}}(X,\Oscr_X)_\mu=0$.
\end{corollary}
\medskip
\begin{proof} Since $T$ acts generically we have $\dim R^T=d-s$.

According to Proposition \ref{prop:weakly} and Lemma \ref{lemma}  $R^T_\mu$ is Cohen-Macaulay if and only if
for
 all $\lambda\in Y(T)$ such that $d_\lambda-h^0_\lambda<d-s$ we have $H^{d_\lambda}_{X^{\lambda,>0}}(X,\Oscr_X)_\mu=0$.

If $\lambda=0$ then $d_\lambda-h^0_\lambda=d-s$ and hence the condition $d_\lambda-h^0_\lambda<d-s$ does not hold.
Assume now $\lambda\neq 0$. We will show that $d_\lambda-h^0_\lambda<d-s$ now always
holds. This proves the lemma.

Assume on the contrary that $d_\lambda-h^0_\lambda\ge d-s$. Since $d_\lambda=d-|T_\lambda|$ this is equivalent to
$|T_\lambda|+h^0_\lambda\le s$. Since $W$ is generic $|T_\lambda|\ge 2$ and in particular
$T_\lambda\neq \emptyset$. Moreover since $T$ acts in particular faithfully we have $h^0_\lambda<s$.
Since $W$ is weakly symmetric  all weights of $W$ not in $H_\lambda$ are in
$\cup_{i\in T_\lambda}\RR\alpha_i$. Fix $f\in T_\lambda$ and
let $\langle \gamma,-\rangle=0$ be a hyperplane in $X(T)_\RR$ containing $H_\lambda$ and $\{\alpha_i\mid i\in T_\lambda, i\neq f\}$ such that $\langle \gamma,\alpha_f\rangle\ge 0$ (this is possible
by the hypothesis $|T_\lambda|+h^0_\lambda\le s$, $h^0_\lambda<s$). Then it
easy to see that there is \emph{at most} one weight such that $\langle \gamma,\alpha_i\rangle >0$ (namely $\alpha_f$). This contradicts the hypothesis that $W$ is generic.
\end{proof}

\begin{remark}\label{rem:1}
Corollary \ref{decor} does not hold true if $W$ is not weakly-symmetric, see \cite[\S4.5]{Vdb1}. See also Remark \ref{rem:2} for an example of a unimodular $W$.
\end{remark}
\begin{corollary}\label{form}
Assume $W$ is generic and weakly symmetric.  
Then the weights of 
$H^{d_\lambda}_{X^{\lambda,>0}}(X,\Oscr_X)_\mu$ are of the form
\begin{equation}
\label{eq:localweights}
-\sum_{i\in T^+_\lambda}\alpha_i-\sum_{i\in T^+_\lambda}a_i\alpha_i+\sum_{i\in T_\lambda}b_i\alpha_i+\sum_{i\in T^0_\lambda} c_i \alpha_i
\end{equation}
$a_i,b_i,c_i\in \ZZ$, $a_i>0$, $b_i>0$.
Thus $R_\mu^T$ is a Cohen-Macaulay $R^T$-module if and only if $\mu$ is not of the form \eqref{eq:localweights}
for all $\lambda\in Y(T)\setminus\{0\}$. 
\end{corollary}

\begin{proof}
This is a consequence of \eqref{loccohstrata} and 
Corollary \ref{decor}. We use the fact that since $W$ is weakly generic the positive integral linear combinations of $(\alpha_i)_{i\in T^0_\lambda}$ form a lattice.
\end{proof}
\subsection{Support in $X^{\lambda,+}$} 
  We will study the local
cohomology of $\Oscr_X$ supported in $X^{u,\chi}$  inductively using the HKKN stratification.
In this section we prove the relevant vanishing theorem. 
More precisely it will follow from Corollary  \ref{local} below that if $R_\mu^T$ is Cohen-Macaulay
then  for all HKKN-strata $S$ and all $i\ge 0$ we have
  $H^i_{S}(X,\Oscr_X)_\mu^T=0$
where we follow the convention that if $S$ is locally closed in $X$ then $H^*_S(X,-):=H^*_S(U,-)$ where $U$ is an open subset of $X$  such that $S$ is closed
in $U$. By excision this definition does not depend on the choice of $U$.
\begin{lemma}\label{lemma:>+} 
We have as $T$-representations:
\begin{align}\label{dual}
H^{i}_{X^{\lambda,+}}(X,\Oscr_X)&=(\wedge^d W)^\ast\otimes_k H^{d-i}_{X^{-\lambda,>0}}(X,\Oscr_X)^*,\\\label{nic}
H^{i}_{X^{\lambda,+}}(X,\Oscr_X)&=0\quad \text{ for $i\neq d-d_{-\lambda}$}.
\end{align}
\end{lemma}
\begin{proof}
This follows from the fact that the defining ideals of $X^{\lambda,+}$ and $X^{-\lambda,>0}$ are generated by complementary subspaces of $W$ together with \cite[Proposition 3.3.1]{Vdb1}. 
\end{proof}
For $f\in R$ we write $X_f=\{f\neq 0\}\subset X$, 
\begin{corollary}\label{local}
Assume $W$ is generic and weakly symmetric. Assume $\lambda\neq 0$ and 
let $f=\prod_{j\in J}x_j$ for a subset $J\subset T_\lambda^0$. 
If $R_\mu^T$ is Cohen-Macaulay $R^T$-module then $R\Gamma_{X^{\lambda,+}_f}(X,\Oscr_{X})_\mu^T=0$.
\end{corollary}

\begin{proof} Let $\delta=\sum_i \alpha_i$ be the character of $\wedge^ d W$.
As
\[
H^i_{X^{\lambda,+}_f}(X,\Oscr_{X})\cong H^i_{X^{\lambda,+}_f}(X_f,\Oscr_{X_f})\cong R_f\otimes_R H^i_{X^{\lambda,+}}(X,\Oscr_{X}),
\]
we see that 
$H^i_{X^{\lambda,+}_f}(X,\Oscr_{X})=0$ for $i\neq d-d_{-\lambda}$  by \eqref{nic} and for $i=d-d_{-\lambda}$  the weights  $\mu'$ of $R_f\otimes_R H^{d-d_{-\lambda}}_{X^{\lambda,+}_f}(X,\Oscr_{X})$  are of the form 
\begin{equation}
\label{eq:cond1}
\mu'=-\delta+\sum_{i\in T^+_{-\lambda}}\alpha_i+\sum_{i\in T^+_{-\lambda}}a_i\alpha_i-\sum_{i\in T_{-\lambda}}b_i\alpha_i+\sum_{i\in T^0_\lambda} c_i \alpha_i
\end{equation}
$a_i,b_i,c_i\in \ZZ$, $a_i>0$, $b_i>0$  by \eqref{eq:localweights}.

On the other hand since $W$ is generic we have by \cite{Knop} $\omega_{R^T}=R^T_{-\delta}$ and since $\Hom_{R^T}(R^T_\mu,\omega_{R^T})=R^T_{-\mu-\delta}$
by \cite[Lemma 4.1.3]{SVdB} as $W$ is generic, $R^T_{-\mu-\delta}$ is also a Cohen-Macaulay $R^T$-module. Hence by applying
\eqref{eq:localweights} for $-\mu-\delta,-\lambda$ we find
\begin{equation}
\label{eq:cond2}
-\mu-\delta\neq -\sum_{i\in T^+_{-\lambda}}\alpha_i-\sum_{i\in T^+_{-\lambda}}a_i\alpha_i+\sum_{i\in T_{-\lambda}}b_i\alpha_i+\sum_{i\in T^0_\lambda} c_i \alpha_i
\end{equation}
for all $a_i,b_i,c_i\in \ZZ$, $a_i>0$, $b_i>0$. Clearly \eqref{eq:cond2} implies that  \eqref{eq:cond1} cannot be satisfied for $\mu'=\mu$.
\end{proof}

\subsection{Support in the $\chi$-unstable locus}\label{sec:chiunstable}
In this section we proceed to compute the local cohomology supported in the $\chi$-unstable locus $X^{u,\chi}$ using the HKKN stratification and applying the vanishing results from the previous section. We show that  $H^*_{X^{u,\chi}}(X,\Oscr_X)_\mu^T$ vanishes if $R_\mu^T$ is Cohen-Macaulay.

We first recall some properties of the HKKN stratifications that we will use. 

Let $(S_i)_{i=1}^N$  be the {\em HKKN stratification}  of $X^{u,\chi}$ (see \cite{Kirwan} and \cite{BFK} for its application in a similar context). It satisfies the following properties:
\begin{enumerate}
\item
The closure of $S_i$ is contained in $\bigcup_{j\ge i}S_j$. 
\item If $S$ is an HKKN stratum then 
$S$ is an open subset of $X^{\lambda,+}$ for some $\lambda\neq 0$ by the proof of \cite[Corollary 13.2]{Kirwan}. Moreover, $S$ is the intersection of $X^{\lambda,+}$ with a union 
$U$ of open sets of the form $X_{f_J}$ where $f_J=\prod_{j\in J}x_j$ for some $J\subset T_{\lambda}^0$ (see \cite[Definition 12.20]{Kirwan}).
\end{enumerate}
\begin{lemma}\label{HKKNstrata} Assume $W$ is generic and weakly symmetric. Let $S$ be an HKKN stratum in $X^{u,\chi}$.
If $R_\mu^T$ is a Cohen-Macaulay $R^T$-module 
then $R\Gamma_{S}(X,\Oscr_{X})_\mu\allowbreak=0$.
\end{lemma}
\begin{proof} Let the notations be as in (2) above.
As $X^{\lambda,+}$ is closed in $X$, 
$S$ is closed in $U$. Hence by definition $H^i_{S}(X,\Oscr_X)\cong H^i_{S}(U,\Oscr_{U})$. 
We have 
\[
R\Gamma_{S}(U,\Oscr_{U})=R\Gamma (U,\Hscr^{c_\lambda}_{X^{\lambda,+}}(X,\Oscr_X))[-c_\lambda],
\]
where $c_\lambda=d-d_{-\lambda}$
and hence
\begin{equation}
\label{cech}
H^i_{S}(U,\Oscr_{U})=H^{i-c_\lambda}(U,\Hscr^{c_\lambda}_{X^{\lambda,+}}(X,\Oscr_X)),
\end{equation}
Recall that $U$ has a covering consisting of open sets of the form
$X_{f_J}$, $J\subset T^0_\lambda$. We can compute the right-hand side of \eqref{cech} using the \v Cech complex with respect to this covering  noting that 
$X_{f_{J_1}}\cap X_{f_{J_2}}=X_{f_{J_1\cup J_2}}$.
To prove the lemma it is then sufficient to prove that $H^{i-c_\lambda}(X_{f_J},\Hscr^{c_\lambda}_{X^{\lambda,+}}(X,\Oscr_X))_\mu=0$.
This follows from Corollary \ref{local}.
\end{proof}
\begin{proposition}\label{faithful} Assume $W$ is generic and weakly symmetric.
If $R_\mu^T$ is a Cohen-Macaulay $R^T$-module then the natural map
\[
R^T_\mu\r R\Gamma(X^{ss,\chi},\Oscr_X)_\mu^T
\]
is an isomorphism.
\end{proposition}

\begin{proof}
Put $S_0=X^{ss,\chi}$ and $X_i:=\bigcup_{j\leq i} S_j$. $X^{ss,\chi}=X_0\subset X_1\subset\cdots\subset X_N=X$ is a filtration
of $X$ by open subsets such that $S_l=X_l-X_{l-1}$ is closed in $X_l$.
We thus have a distinguished triangle 
\[
R\Gamma_{S_l}(X_l,\Oscr_{X_l})\to R\Gamma(X_l,\Oscr_{X_l})\to R\Gamma(X_{l-1},\Oscr_{X_{l-1}})\to
\]
Since $R\Gamma_{S_l}(X_l,\Oscr_{X_l})^T_\mu=
R\Gamma_{S_l}(X,\Oscr_{X})^T_\mu=0$ by Lemma \ref{HKKNstrata}
we have $R\Gamma(X_l,\Oscr_{X_l})^T_\mu\allowbreak\cong R\Gamma(X_{l-1},\Oscr_{X_{l-1}})^T_\mu$.
Thus,  $R\Gamma(X^{ss,\chi},\Oscr_{X^{ss,\chi}})_\mu^T\cong R\Gamma(X,\Oscr_X)_\mu^T=R^T_\mu$.
\end{proof}

\begin{remark}\label{rem:2}
Proposition \ref{faithful} does not hold without the weak symmetry assumption.  
The reason goes back to Corollary \ref{decor}, see Remark \ref{rem:1}. 
It fails for instance in the example \cite[\S4.5]{Vdb1} (for $\mu=(-3,-3)$, $\chi=(-2,-1)$). Slightly tweaking the example in loc. cit. we can also get a unimodular $W$, which we mention here but omit the details.  
Let $T=G_m^2$ and let $W$ be with weights $(1,0)$, $(2,0)$, $(0,1)$, $(-1,1)$, $(-1,-1)^{\oplus 2}$, $\mu=(-3,-1)$, $\chi=(1,-2)$.
 By \cite{Vdb1} (c.f. \S\ref{nullcone}), $R_{\mu}$ is Cohen-Macaulay, however it is easy to verify that $H^3_{X^{(2,1),>0}}(X,\Oscr_X)_{\mu}\neq 0$ (c.f. \ref{loccohstrata}) and thus also $H^3_{X^{(2,1),+}}(X,\Oscr_X)_{\mu}\neq 0$ by Lemma \ref{lemma:>+}. One can check (using the Mayer-Vietoris sequence) that $H^3_{X^{(2,1),+}}(X,\Oscr_X)$ occurs in $H^3_{X^{u,\chi}}(X,\Oscr_X)$. Thus, Proposition \ref{faithful} does not hold in this case.
\end{remark}

\section{GIT quotient stacks vs NCCRs}
\label{sec:sec5}
In this section we show that any NCCR is derived equivalent to $X^{ss,\chi}/T$, which also proves Theorem \ref{thm}. 
For a family of specific NCCRs in the quasi-symmetric case, 
that were constructed in \cite[Theorem 1.6.2.]{SVdB}, 
this result had been established in \cite[Corollary 4.2, Remark 4.3]{HLSam}.

\begin{proposition}\label{prop:git}{}
Assume that $W$ is unimodular, generic and  weakly symmetric.
Let $\Lambda$ be a toric NCCR of $R^T$, and let $\chi$ be a generic character of $T$. 
Then $\D(\Lambda)\cong \D(X^{ss,\chi}/T)$.
\end{proposition}

\begin{proof}
Let $\Lambda=\End_{R^T}(M(U))$ for $U=\bigoplus_{i\in I}\chi_i$. 
Put $\mathcal{E}=\bigoplus_{i\in I}\chi_i \otimes_k\Oscr_{X^{ss,\chi}}$. 
Note that $\Lambda\cong M(\End(U))\cong \bigoplus_{i,j\in I}M(\chi_i-\chi_j)$ as $W$ is generic (see \cite[Lemma 4.1.3]{SVdB}). 
Proposition \ref{faithful} shows that the functor
\[
 -\Lotimes_\Lambda \mathcal{E}: \D(\Lambda) \to \D(X^{ss,\chi}/T)
\]
is fully faithful. As $\gldim \Lambda<\infty$, $\D(\Lambda)$ is an admissible subcategory in $\D(X^{ss,\chi}/T)$.  Since $\chi$ is generic
$X^{ss,\chi}/T$ is a Deligne-Mumford stack (see e.g. \cite[Proposition 2.1]{HLSam}, the proof goes through under the weak symmetry assumption).
We now use \cite[Corollary A.5]{SVdB5} 
 which asserts that $\D(X^{ss,\chi}/T)$ has no nontrivial semi-orthogonal decomposition. 
 Thus, $\D(\Lambda)\cong \D(X^{ss,\chi}/T)$.
\end{proof}

\begin{proof}[Proof of Theorem \ref{thm}]
The result follows immediately from Proposition \ref{prop:git}.
\end{proof}

\section{Example}\label{sec:example}
If $\Delta\subset\RR^n$ is a bounded closed convex polygon and $\varepsilon\in \RR^n$ then $\Delta_\varepsilon=\bigcup_{r>0} \Delta\cap (r\varepsilon+\Delta)$. 
Assume $W$ is quasi-symmetric and generic. The NCCRs that were
constructed in \cite[Theorem 1.6.2.]{SVdB} are given by modules of{}
covariants $M(U)$ where $U$ is the sum of the characters contained in
$1/2\bar{\Sigma}_\varepsilon\cap X(T)$ for a generic (in the sense of Definition \ref{def:generic})
$\varepsilon\in X(T)$. 
One can check that the proofs hold true also if we replace $1/2\bar{\Sigma}_\varepsilon$ by $\nu+1/2\overline{\Sigma}$ for $\nu\in X(T)_\RR$ such that $(\nu+(1/2)\partial\overline{\Sigma})\cap X(T)=\emptyset$.

In this section we give an example which shows that not all toric NCCRs come from $(\nu+1/2\overline{\Sigma})\cap X(T)$ (for $\nu$ as above), so Proposition \ref{prop:git} does not follow from \cite[Corollary 4.2, Remark 4.3]{HLSam}. 

We take $T=G_m$ and let $-3,-2,-2,2,2,3$ be the weights of $W$. The set of Cohen-Macaulay modules of covariants is given by $\{i\mid -7<i<7\}\cup\{-8,8\}$,  which can be deduced from \cite{Vdb1} (c.f. \S\ref{nullcone}).  Note that $\Sigma=(-7,7)$. Let $U$ have  weights  $-4,-2,-1,0,1,2,4$.  This set of weights is not an interval and hence it is not of the form $(\nu+
1/2\bar{\Sigma})\cap X(T)$. However, it is a maximal Cohen-Macaulay clique (see \eqref{1} below).
Thus Proposition \ref{prop:MMA} below implies that $\Lambda=\End_{R^T}(M(U))$ is an NCCR of $R^T$.

\section{Toric NCCRs in the case $T=G_m$}\label{app}
In this section we give an explicit combinatorial criterion (based on Appendix \ref{sec:Jason} written by Jason Bell)   for recognizing toric NCCRs in the case $T=G_m$, and prove that they are all related
by a ``mutation'' procedure, which in particular gives a new proof of Theorem \ref{thm} in this case. 
Meanwhile we obtain some  relations  between NCCRs and ``maximal modification algebras'' which we recall first.
\begin{definition}\cite{IW}\label{def:MMA} 
Let $S$ be a noetherian Cohen-Macaulay ring. 
A reflexive $S$-module $M$ is {\em modifying} if $\End_S(M)$ is Cohen-Macaulay. It is {\em maximal modifying} (MM) if it is modifying and if $M\oplus M'$ is modifying for a reflexive module $M'$ then $M'\in {\rm add M}$ (i.e. $M'$ is a direct summand of direct sums of $M$). If $M$ is an MM module, then $\End_S(M)$ is a {\em maximal modification algebra} (MMA). 
\end{definition} 
We consider a toric variant of this definition. We say that
$\End_{R^T}(M)$ is a {\em toric MMA} if $M$ is a module of covariants
and maximal modifying with respect to modifying modules of covariants.
Hence an MMA is automatically a toric MMA but the converse might
not hold. See \cite[Example 10.1]{SVdB}, \cite[Example 3.3]{SVdB4} for
a (counter)example.

\medskip

An NCCR is always an MMA by
\cite[Proposition 4.5]{IW} and the converse fails in
general (see \cite[Example A.1]{VdB04a} and \cite[Theorem 4.16, Remark 4.17]{IW1}).
Surprisingly, for toric MMAs and NCCRs in the case $T=G_m$ there is no
distinction.

\begin{proposition} (see \S\ref{sec:proof} below) \label{prop:MMA}
If $T=G_m$ and $W$ is a generic and unimodular $T$-representations then  
toric MMAs and toric NCCRs of $\Sym(W)^T$ coincide.
\end{proposition} 
We next describe our strategy for an alternative proof of Theorem \ref{thm}. Note that we may and we will assume that $W$ has no zero weights since extra zero weights do not affect the NCCR property.
\begin{enumerate}
\item\label{1}  
We say that $S\subset X(T)\cong \ZZ$ is a {\em Cohen-Macaulay clique} \cite{Bocklandt} if for every $i,j\in S$ the module of covariants $M(i-j)$ is Cohen-Macaulay. 
Note that toric MMAs correspond precisely to {maximal Cohen-Macaulay cliques}. 
A combinatorial argument provided by Jason Bell (see Appendix \ref{sec:Jason}) 
shows that each maximal Cohen-Macaulay clique contains exactly one element congruent to $i$ for $0\leq i<N:=\sum_{\alpha_i>0}\alpha_i$.

\item\label{2} Using the complexes connecting projective
  $(T,R)$-modules \cite[\S11]{SVdB5} and \eqref{1} we show that for a
  toric MMA $\Lambda=\End_{R^T}(M)$ and for $\Lambda'=\End_{R^T}(M')$,
  where $M'=M(U')$ is obtained from $M=M(U)$ by replacing the highest
  weight $\mu_{\max}$ of $U$ by $\mu_{\max}-N$, the
  $(\Lambda,\Lambda')$-bimodule $\Hom_{R^T}(M',M)$ defines a derived
  equivalence between $\Lambda'$ and $\Lambda$. 
\item By induction on the maximum difference between the weights of $U$ we can therefore, using \eqref{1}, construct a derived equivalence between 
the ``standard'' NCCR $\End_{R^T}(\oplus_{0\leq \mu<N}M(\mu))$ \cite[Theorem 8.9]{VdB04} and a toric MMA $\End_{R^T}(M(U))$. 
\end{enumerate} 

\subsection{Toric MMAs}
In this section we deduce from Appendix \ref{sec:Jason} that all maximal Cohen-Macaulay cliques have the same size, and moreover that they have a  particular form which will be of vital importance in \S\ref{sec:IW}. 

We write $T^+=T_1^+=\{\alpha_i>0\}$, $T^-=T_1^-=\{\alpha_i<0\}$. 
Setting $\{a_i\}_i=\{\alpha_i\}_{i\in T^+}\cup\{-\alpha_i\}_{i\in T^-}$ and $N=\sum_{i\in T^+}\alpha_i=-\sum_{i\in T^-}\alpha_i$ we deduce from Lemma \ref{nullcone} and Theorem \ref{theorem} that $\Sscr=\Sscr_+\cup\Sscr_-$ in Appendix \ref{sec:Jason} corresponds to the set of non-Cohen-Macaulay weights (i.e. weights $\mu$ such that $M(\mu)$ is not Cohen-Macaulay).

\begin{remark}\label{rem:trivial}
Note that $\Sscr=-\Sscr$ so $\Mscr$ is a Cohen-Macaulay clique if for all 
$i,j\in \Mscr$ we have $M(i-j)$ or $M(j-i)$ is Cohen-Macaulay.
\end{remark}

The following corollary is an immediate consequence of Corollary \ref{cor:size}.

\begin{corollary}\label{cor:cong:i}
If $\T$ is a maximal Cohen-Macaulay clique, then for every $0\leq i<N$ there exists a unique $m\in \T$ such that $m\equiv i\,(N)$. 
\end{corollary}

For the construction of derived equivalence via ``mutation''  we will also need the following easy corollary.
\begin{corollary}\label{cor:mutation}
Let $\T$ be a maximal Cohen-Macaulay clique, let $m_{\max}=\max \T$ and let $\emptyset \neq S\subseteq T^-$. Then $m_{\max}+\sum_{i\in S}\alpha_i\in\T$ and $(\T\setminus \{m_{\max}\})\cup \{m_{\max}-N\}$ is a maximal Cohen-Macaulay clique.
\end{corollary}

\begin{proof}
By Corollary \ref{cor:cong:i} and the maximality assumption we have that $m:=m_{\max}+\sum_{i\in S}\alpha_i-kN\in \T$ for some $k\geq 0$. 
Since $m_{\max}-m=kN-\sum_{i\in S}\alpha_i\not\in \Sscr^+$ we must have $k=0$. 

For the second claim we need to show that $-N<m_{\max}-N-m\not\in\Sscr_+$ for $m\in \T$ by Remark \ref{rem:trivial}.   
If $m_{\max}-N-m\in \Sscr_+$ then also $m_{\max}-m=(m_{\max}-N-m)+N\in \Sscr_+$, a contradiction.
\end{proof}
\subsection{Derived equivalence}\label{sec:IW}
Let $\Lambda=\End_{R^T}(M(U))$ be a toric MMA. 
Up to Morita equivalence we may, and will, assume  that every weight in $U$ occurs with multiplicity
$1$. Let $\mu_{max}$ be the maximal weight of $U$ and let $U'$ be a
representation given by replacing $\mu_{\max}$ in $U$ by
$\mu_{\max}-N$.
We will refer to $\Lambda'=\End_{R^T}(M(U'))$ as a {\em mutation} of
$\Lambda$.

In this section we show that a toric MMA and its mutation are derived
equivalent. Repeating the mutation we obtain that a toric MMA is
derived equivalent to the ``standard'' NCCR, providing an alternative
proof of Theorem \ref{thm}.
\subsubsection{Derived equivalence of a toric MMA and its mutation}
Let $\Lambda$ be a noetherian ring. A finitely generated
$\Lambda$-module $X$ is {\em tilting} if 
\begin{enumerate}
\item\label{t1}
$\pdim_{\Lambda}X<\infty$,
\item\label{t2}
$\Ext^i_\Lambda(X,X)=0$ for $i>0$,
\item\label{t3}
$X$ is a generator of $D(\Lambda)$.
\end{enumerate} 
It is a classical result that $\Lambda$ and $\End_\Lambda(X)$ are derived equivalent \cite{happel}. 
\begin{proposition}\label{prop:tilting}
Let notation be as above and denote $M=M(U)$, $M'=M(U')$. 
The $(\Lambda,\Lambda')$-bimodule $X=\Hom_{R^T}(M',M)$ is a {\em tilting} $\Lambda'$-module. In particular, $\Lambda'$ and $\Lambda=\End_{\Lambda'}(X)$ are derived equivalent.
\end{proposition}

\begin{proof}{}
We will verify the tilting conditions \eqref{t1}-\eqref{t3} by employing complexes used in \cite[\S 11.2]{SVdB} for constructing NCCRs (and also known as a part of Weyman's geometric method \cite{weyman2003cohomology}). 
Let $K_{-}={\rm span}\{w_i\mid\alpha_i<0\}$,  $K_+={\rm span}\{w_i\mid \alpha_i>0\}$,
 $d^\pm=\dim K_{\pm}=|T^{\pm}|$. 
The complexes are obtained from the Koszul resolutions $\Kscr_\pm$ of $R_\pm=\Sym(W/K_{\mp})$,   which remain exact after tensoring with $\chi\in X(T)$: 
\begin{equation}\label{eq:complex-}0\r \chi\otimes_k\wedge^{d^\mp}K_{\mp1} \otimes_k R\r \cdots\r \chi\otimes_kK_{\mp1}\otimes_k R\r \chi\otimes_k R\r \chi\otimes_k R_\pm\r 0.\end{equation}

To show \eqref{t1} let us recall that the indecomposable projective
right $\Lambda'$-modules are of the form $\Hom_{R^T}(M',M(\mu))$ where
$\mu$ is a weight of $U'$. Let $\mu_1< \dots< \mu_\ell=\mu_{\max}$ be
the weights of $U$.  Since
$X=\oplus_{i<\ell}\Hom_{R^T}(M',M(\mu_i))\oplus
\Hom_{R^T}(M',M(\mu_\ell))$, it is enough to show that $\pdim_{\Lambda'}
\Hom_{R^T}(M',M(\mu_\ell))$ is finite.

We use \eqref{eq:complex-}  with $\Kscr_+$ and $\chi=\mu_\ell$. 
Note that $\Hom_{R}(U'\otimes_k R,\mu_\ell\otimes_k R_+)^T=0$ since 
otherwise $-\mu+\mu_\ell+\sum_{i\in T^+}a_i\alpha_i=0$ for some weight $\mu$ of $U'$ and $a_i\in \NN$, and consequently $\mu-\mu_\ell+N=N+\sum_{i\in T^+}a_i\alpha_i\in \Sscr_+$, which contradicts the fact that the set of weights of $U'$ is a Cohen-Macaulay clique by Corollary \ref{cor:mutation}. Thus, applying $\Hom_{R}(U'\otimes_k R, -)^T$ to this complex we obtain, using Corollary \ref{cor:mutation}, a
$\Lambda'$-projective resolution $Q^\bullet$ of $\Hom_{R^T}(M',M(\mu_\ell))$. Hence \eqref{t1} follows. Moreover, from $Q^\bullet$ we also obtain \eqref{t3}. 

For \eqref{t2} we need to show that $\RHom_{\Lambda'}(X,X)$ has cohomology only in degree $0$. Since $\Hom_{R^T}(M',M(\mu_i))$, $i<\ell$, are projective $\Lambda'$-modules, it suffices to show that $\RHom_{\Lambda'}(\Hom_{R^T}(M',M(\mu_\ell)),\Hom_{R^T}(M',M(\mu_i)))$ 
has cohomology only in degree $0$. 
We can replace $\Hom_{R^T}(M',M(\mu_\ell))$ by 
$Q^\bullet=\Hom_{R^T}(M',\sigma_{<0}(\Kscr_+\otimes_k \mu_\ell[-1])^T)$. We have
\begin{align*}
\RHom_{\Lambda'}&(\Hom_{R^T}(M',M(\mu_\ell)),\Hom_{R^T}(M',M(\mu_i)))\\&= 
\Hom_{\Lambda'}(Q^\bullet,\Hom_{R^T}(M',M(\mu_i)))\\
&= 
\Hom_{\Lambda'}(\Hom_{R^T}(M',\sigma_{<0}(\Kscr_+\otimes_k \mu_\ell[-1])^T),\Hom_{R^T}(M',M(\mu_i)))\\
&=\Hom_{R^T}(\sigma_{<0}(\Kscr_+\otimes_k \mu_\ell[-1])^T,M(\mu_i)))\\
&=\sigma_{>0}(\Hom_{R^T}((\Kscr_+\otimes_k \mu_\ell)^T,M(\mu_i))[1])\\
&=\sigma_{>0}\Hom_R(\Kscr_+,R\otimes_k (-\mu_\ell+\mu_i))^T[1]\\
&=\sigma_{>0}(\wedge^{d_-}{K_-^*}\otimes_k\Kscr_+[-d] \otimes_k (-\mu_\ell+\mu_i))^T[1]
\end{align*}
which is exact in degrees $>0$ since $R_+$ is the cohomology of $\Kscr_+$ and 
\[
((\wedge^{d_-}K_-)^*\otimes_k R_+\otimes_k (-\mu_\ell+\mu_i))^T=(R_+\otimes_k (N-\mu_\ell+\mu_i))^T=0
\]
as otherwise $\sum_{i\in T^+}a_i\alpha_i+N-\mu_\ell+\mu_i=0$ for some $a_i\in \NN$ which contradicts the fact that $\mu_\ell-\mu_i\not\in \Sscr_+$ (as the set of weights of $U$ is a Cohen-Macaulay clique).
\end{proof}

\subsubsection{Proof of Theorem \ref{thm} and Proposition \ref{prop:MMA}}\label{sec:proof}
Let $\Lambda=\End_{R^T}(M(U))$ be a toric MMA. Recall that every (toric) NCCR is a (toric) MMA. By Morita equivalence we can assume that every weight of $U$ appears with multiplicity $1$.  Let $\Lscr$ be the set of weights of $U$.
By translation we can assume $0\in\Lscr\subset\NN$.  We argue by
induction on $\max \Lscr$ that $\Lambda$ is derived equivalent to
$\Lambda_0=\End_{R^T}(M(U_0))$ where the weights of $U_0$ are given by
$\Lscr_0=[0,N-1]$.

By Corollary \ref{cor:cong:i}, $\max \Lscr\geq N-1$, and if $\max\Lscr=N-1$ then $\Lambda=\Lambda_0$. Thus, we can assume that $\max \Lscr\geq N$. 
Let $\Lambda'=\End_{R^T}(U')$ be a mutation of $\Lambda$, and let $\Lscr'$ be set of weights of $U'$. 
Then $\Lambda'$ is a toric MMA by Corollary \ref{cor:mutation}, with $0\in \Lscr'\subset \NN$ and $\max \Lscr'<\max \Lscr$. 
By induction, $\Lambda'$ is derived equivalent to $\Lambda_0$. 
We can use Proposition \ref{prop:tilting} to conclude that  $\Lambda$ is derived equivalent to $\Lambda'$ and thus to $\Lambda_0$. In particular, $\Lambda$ 
is an NCCR, proving also Proposition \ref{prop:MMA}.

\appendix
\section{Appendix by Jason P. Bell}\label{sec:Jason}
By \S\ref{nullcone}, understanding the maximal Cohen-Macaulay cliques reduces to a purely combinatorial problem which is a subject of this section. 
Let $a_1,\ldots ,a_d\in \NN_{>0}$. 
In addition, we assume that ${\rm gcd}(a_i)=1$. 
As a consequence
every sufficiently large natural number can be expressed as a linear combination of the $a_i$
 with nonnegative integer coefficients. 
 Let $N\in \NN$. 
We define
\begin{equation*}
\mathcal{S}_+:=\left\{ N + \sum_i c_i a_i \colon c_i\ge 0\right\},\quad\mathcal{S}_{-}=-\mathcal{S}_{+},\quad \mathcal{S}=\mathcal{S}_+\cup \mathcal{S}_-.
\end{equation*}
Then $\mathcal{S}$ contains all but finitely many integers.

Let 
\[
\TT=\{\T\subset \ZZ\mid m,m'\in \T\implies m-m'\not\in \mathcal{S}\}.
\]
Given $i\in \{0,\ldots ,N-1\}$, we define $p(i)$ to be the smallest positive integer $p$ such that $i+pN$ is in $\mathcal{S}_+$.  Similarly, $q(i)$ is the largest negative integer $q$ such that $i+qN$ is in $\mathcal{S}_{-}$.  Let $\T\in \TT$ and let $j\in \mathbb{Z}\setminus \T$.  We say that $m\in \T$ \emph{blocks} $j$ if $j-m\in \mathcal{S}$.  Notice that if $\T$ is a maximal element of $\TT$ then for each element in the complement of $\T$ there is necessarily some element of $\T$ that blocks it.
\begin{lemma}
\label{lem: i}
Let $\T\in\TT$ be a maximal element of $\TT$ containing $0$ and let $i\in \{0,\ldots ,N-1\}$.  Then $\T$ contains an element that is congruent to $i$ modulo $N$.
\end{lemma}
\begin{proof}
Suppose not.  Then for each $j\in \{q(i),\ldots ,p(i)\}$ we can choose some integer $m_j\in \T$ such that $m_j$ blocks $i+jN$.
Since both $i+p(i)N$ and $i+q(i)N$ are in $\mathcal{S}$, we can take $m_{p(i)}=m_{q(i)}=0$.  
Now let
$$X_{\pm} = \{j\in \{q(i),\ldots ,p(i)\} \colon m_j - (i+jN) \in \mathcal{S}_{\pm}\}.$$ 
Then $X_{+}$ and $X_{-}$ are disjoint and their union is all of $\{q(i),\ldots ,p(i)\}$.  Moreover, since $m_{p(i)}=m_{q(i)}=0$ and $q(i)<0<p(i)$, we have
$q(i)\in X_{+}$ and $p(i)\in X_{-}$.  In particular, there must exist some $j\in \{q(i),\ldots ,p(i)-1\}$ such that
$j\in X_{+}$ and $j+1\in X_{-}$.
Given such a $j$, we then have
$$m_j - (i+jN)\in \mathcal{S}_{+}\qquad {\rm and}\qquad m_{j+1} - (i+(j+1)N)\in \mathcal{S}_{-}.$$
So we may write $m_j - (i+jN) = N + k_1$ and
$m_{j+1}-(i+(j+1)N) = -N-k_2$, where $k_1$ and $k_2$ are $\NN$-linear combinations of the $a_i$. 
  Subtracting these two equalities, we see
$$m_j - m_{j+1}+N = \left( m_j - (i+jN) \right) - \left(m_{j+1}-(i+(j+1)N)\right) = 2N + k_1+k_2.$$
In particular,
$m_j-m_{j+1}= N+k_1+k_2\in \mathcal{S}_+$.  But this contradicts the fact that $m_j,m_{j+1}\in \T\in \TT$.  The result follows.
\end{proof}
\begin{corollary}\label{cor:size}
Let $N$ be a $\NN$-linear combination of $a_i$. 
Let $\T$ be a maximal element in $\TT$. 
For every $0\leq i<N$ there exists exactly one element $m\in \T$ such that $m\equiv i \,(N)$. In particular, all maximal elements of $\TT$ have size $N$.
\end{corollary}
\begin{proof}
Let $\T$ be a maximal element of $\TT$. By translation we can assume that $0\in \T$. 
By Lemma \ref{lem: i} we have that for each $i\in \{0,\ldots ,N-1\}$ there is some element of $\T$ that is congruent to $i$ mod $N$.  Thus $|\T|\ge N$.  On the other hand, by definition of $\mathcal{S}_{+}$ and the assumption on $N$ we have $\{N,2N,3N,\ldots \}\subseteq \mathcal{S}$.   If $\T$ had size strictly larger than $N$ then there would exist $m,m'\in \T$ with $m> m'$ and $m$ and $m'$ congruent to $0$ mod $N$.  But then $m-m'$ is a positive multiple of $N$ and hence in $\mathcal{S}$.  
\end{proof}

\bibliographystyle{amsalpha}

\end{document}